\documentclass[12pt,a4paper]{amsart}
\usepackage{mathrsfs}

\newtheorem{theo+}              {Theorem}           [section]
\newtheorem{prop+}  [theo+]     {Proposition}
\newtheorem{coro+}  [theo+]     {Corollary}
\newtheorem{lemm+}  [theo+]     {Lemma}
\newtheorem{exam+}  [theo+]     {Example}
\newtheorem{rema+}  [theo+]     {Remark}
\newtheorem{defi+}  [theo+]     {Definition}
\newtheorem{clai+}  [theo+]     {Claim}
\newenvironment{theorem}{\begin{theo+}}{\end{theo+}}
\newenvironment{proposition}{\begin{prop+}}{\end{prop+}}
\newenvironment{corollary}{\begin{coro+}}{\end{coro+}}
\newenvironment{lemma}{\begin{lemm+}}{\end{lemm+}}

\usepackage{amsthm}
\theoremstyle{plain} \theoremstyle{remark}
\newtheorem{remark}{Remark}

\def \r{\mbox{${\mathbb R}$}}
\def\E{/\kern-1.0em \equiv }

\title{  Biharmonic  Riemannian submersions  from
 the product space  $M^2\times\r$}

\author{Ze-Ping Wang$^{*}$ and Ye-Lin Ou$^{**}$ }

\address{School of Mathematical Sciences,\newline\indent Guizhou
Normal University,\newline\indent Guiyang 550025,\newline\indent
People's Republic of China
\newline\indent E-mail:zpwzpw2012@126.com \;(Wang) \\ \newline
\newline\indent Department of
Mathematics,\newline\indent Texas A $\&$ M University-Commerce,
\newline\indent Commerce TX 75429,\newline\indent USA.\newline\indent
E-mail:yelin$\_$ou@tamu-commerce.edu \;(Ou).}

\thanks{*Supported by the Natural Science Foundation of China (No. 11861022). \\
\indent** Supported by a grant from the Simons Foundation ( 427231,
Ye-Lin Ou).}
\begin{document}
\title[Biharmonic  Riemannian submersions ] {Biharmonic  Riemannian submersions  from
 the product space  $M^2\times\r$}
\date {2/24/2024} 
\maketitle

\section*{Abstract}
\begin{quote}
{\footnotesize In this paper,  we study  biharmonic Riemannian submersions $\pi:M^2\times\r\to (N^2,h)$ from a product manifold onto a surface and obtain some local characterizations of such biharmonic maps. Our results show that when the target surface is flat, a proper biharmonic Riemannian submersion $\pi:M^2\times\r\to (N^2,h)$ is locally a projection of a special twisted product, and when the target surface is non-flat, $\pi$ is locally a  special map between two warped product spaces with a warping function that solves a single ODE. As a by-product, we also prove that there is a unique proper biharmonic Riemannian submersion $H^2\times \r\to \r^2$ given by the projection of a warped product.}
\end{quote}
\section{Introduction and preliminaries}

A {\bf harmonic  map}  is a  map  $\varphi:(M, g)\to (N,
h)$ between Riemannian manifolds whose tension filed  vanishes identically, i.e., $\tau(\varphi)={\rm
Trace}_{g}\nabla {\rm d} \varphi\equiv 0$. A {\bf biharmonic map} is one whose bitension filed solves the PDEs
\begin{equation}\label{BT1}
\tau_{2}(\varphi):={\rm
Trace}_{g}(\nabla^{\varphi}\nabla^{\varphi}-\nabla^{\varphi}_{\nabla^{M}})\tau(\varphi)
- {\rm Trace}_{g} R^{N}({\rm d}\varphi, \tau(\varphi)){\rm d}\varphi
=0,
\end{equation}
where $R^{N}$ is the curvature operator of $(N, h)$ defined by
$$R^{N}(X,Y)Z=
[\nabla^{N}_{X},\nabla^{N}_{Y}]Z-\nabla^{N}_{[X,Y]}Z.$$
Clearly, any harmonic map is a biharmonic map. A biharmonic map which is not harmonic is called a {\bf proper biharmonic map}.\\

 The geometric study of biharmonic maps focuses on biharmonicity of maps with geometric interest like isometric immersions or Riemannian submersions, and the geometry and topology of  the spaces related to the existence of such geometrically special biharmonic maps. For example, in the study of biharmonic submanifolds (biharmonic isometric immersions), a fundamental problem is to classify biharmonic submanifolds in certain model spaces. Although many related progresses have been made the following conjectures remain open in general cases.\\
 
{\bf Chen's Conjecture (\cite{CH}):} Any biharmonic submanifolds in Euclidean space is minimal.\\
 
\noindent {\bf Conjecture (Balmus-Montaldo-Oniciuc \cite{BMO1}):} A biharmonic subman-
ifold in sphere has constant mean curvature; and any proper biharmonic hypersurface in $S^{m+1}$ is an open part of $S^{m}(\frac{1}{\sqrt{2}})$ or the generalized Clifford 
torus $S^{p}(\frac{1}{\sqrt{2}})\times S^{q}(\frac{1}{\sqrt{2}})$ with $p+q=m, p\ne q$.\\

For more detailed information and some recent progress on biharmonic submanifolds see the recent book \cite{Ou7} and the vast references therein.

 Riemannian submersions are a dual concept of  isometric immersions (i.e., submanifolds). The study of biharmonicity of Riemannian submersions was initiated in \cite{Oni}. A useful tool of using the so-called integrability data to study  biharmonic Riemannian submersion from a  generic 3-manifold was introduced  in \cite{WO}. This was later generalized to higher dimensions with one dimensional fibers in \cite{AO}. Complete classifications of Riemannian submersions from a 3-dimensional space forms and  more general BCV  spaces  into a surface were obtained in  \cite{WO,WO1}. 
 
In this paper, we study  biharmonic submersions from the product space $M^2\times\r$, where $M^2$ is a general 2-dimensional manifold.  A reason for the choice of the product space $M^2\times\r$ is that it includes  many well-known model spaces, such as $S^2\times\r$,\; $H^2\times\r$,\; $\r^3$,\; and twisted spaces $\r^2\times_{e^{2p}}\r=(\r^3,dy^2+dz^2+e^{2p(x,y)}dx^2)$.  We give   some local characterizations of biharmonic (including harmonic) Riemannians submersions  from  $M^2\times\r$ onto a surface.  These include: a Riemannian submersion $\pi:M^2\times\r\to (N^2,h)$  is harmonic if and only if it is locally the projection onto the first factor followed by a Riemannian covering map (Theorem \ref{HTh}); A Riemannian submersion $\pi:M^2\times\r\to \r^2$ is  proper biharmonic, then it is locally a projection of a special twisted product (Theorem \ref{Ths1}), and if a Riemannian submersion $\pi:M^2\times\r\to  (N^2,h)$ into a non-flat surface is  proper biharmonic, then $\pi$ is locally a  special map between two warped product spaces with a warping function that solves a single ODE (Corollary \ref{co1}).

\section{ Harmonic and biharmonic Riemannian submersions from $M^2\times\r$}

We will use  the following useful local orthonormal frame in the paper.
\begin{lemma}\label{2.1}
For any point on $M^2\times\r$, there is a neighborhood $U\times \r$ and local coordinates $(t, s, z)$ such that the product metric on $M^2\times\r$ takes the form
$e^{2q(t,s)}dt^2+ds^2+dz^2$. Furthermore, the orthonormal frame
$\{ E_1=e^{-q(t,s)}\partial_t,\;E_2=\partial_s, E_3=\partial_z\}$  satisfies
\begin{equation}\label{g1}
\begin{array}{lll}
[E_1,E_2]=fE_{1},\;\; {\rm all\;\;
other}\;\;[E_i,E_j]=0,\;\;i,j=1, 2, 3,\\
\nabla_{E_{1}}E_{1}=-fE_{2},\;\nabla_{E_{1}}E_{2}=fE_{1},\;
{\rm all\;other}\;\nabla_{E_{i}}E_{j}=0,\;i,\;j=1,2,3,
\end{array}
\end{equation}
where  $f=q_s$, and the only possible
nonzero components of the (Ricci) curvatures of $M^2\times\r$ are given by
\begin{equation}\label{g}
\begin{array}{lll}
{\rm Ric}\, (E_{1},E_{1})={\rm Ric}(E_{2},E_{2})= R_{1212}=g(R(E_{1},E_{2})E_{2},E_{1})=-E_2(f)-f^2,
\end{array}
\end{equation}
which is the Gauss curvature of $M^2$ on $U$ denoted by $K^{M^2}=-E_2(f)-f^2$.
\end{lemma}

\begin{proof} It is well known that around any point of a 2-dimensional Riemannian manifold $(M^2,g)$ there exists a semi-geodesic coordinate system so that locally it can be described as $(U\subseteq M^2, g=e^{2q(t,s)}dt^2+ds^2)$.
One can easily checked that $\{ E_1=e^{-q(t,s)}\partial_t,\;E_2=\partial_s, E_3=\partial_z\}$ is an orthonormal frame satisfies  (\ref{g1}) and (\ref{g}) with $f=q_s$.
\end{proof}

\begin{remark}
Note that the Lemma \ref{2.1}  actually implies that any local semi-geodesic coordinate system gives rise to an orthonormal frame $\{E_i\}$ with the properties (\ref{g1}) and (\ref{g}). For latter reference we call such an orthonormal frame {\em semi-geodesic orthonormal frame}.
\end{remark}

Let $\pi:M^2\times\r \to (N^2,h)$ be a Riemannian
submersion, and let   $\{e_1,\; e_2, \;e_3\}$ be a local orthonormal frame on $W\subset M^2\times\r $  with $e_3$ vertical.  Let $\{f_{1},\;f_2, f_3, \kappa_1,\;\kappa_2,\; \sigma\}$ be the generalized  integrability associated to this frame. Note that $f_{3}=0$ if and only if  the above frame is  adapted  to $\pi$.

Using the relation $e_i=\sum\limits_{j=1}^3a_{i}^{j}E_j$, as it was computed in \cite{WO1}, we have the following  Lie  brackets and the components of the Levi-Civita connection with respect to this frame as
\begin{equation}\label{R2}
\begin{array}{lll}
[e_1,e_3]=f_{3}e_2+\kappa_1e_3,\;
[e_2,e_3]=-f_{3}e_1+\kappa_2e_3,\;
[e_1,e_2]=f_1 e_1+f_2e_2-2\sigma e_3,\\
\nabla_{e_{1}} e_{1}=-f_1e_2,\;\;\nabla_{e_{1}} e_{2}=f_1
e_1-\sigma e_{3},\;\;\nabla_{e_{1}} e_{3}=\sigma
e_{2},\\ \nabla_{e_{2}} e_{1}=-f_2 e_{2}+\sigma
e_3,\;\;\nabla_{e_{2}} e_{2}=f_2 e_{1}, \;\;\nabla_{e_{2}}
e_{3}=-\sigma e_{1},\\ \nabla_{e_{3}}
e_{1}=-\kappa_1e_{3}+(\sigma-f_3) e_{2}, \nabla_{e_{3}} e_{2}= -(\sigma-f_3)
e_{1}-\kappa_2 e_3, \nabla_{e_{3}} e_{3}=\kappa_1 e_{1}+\kappa_2e_2,
\end{array}
\end{equation}
and the only possible nonzero components of the Riemannian curvature $R$ of $M^2\times \r$ as
\begin{equation}\label{RC0}
\begin{cases}
R (e_1,e_3,e_1,e_2)=-e_1(\sigma)+2\kappa_1\sigma=-a_{2}^{3}a_{3}^{3}K^{M^2},\\
R (e_1,e_3,e_1,e_3)=e_1(\kappa_1)+\sigma^2-\kappa_{1}^2+\kappa_2f_1=(a_{2}^{3})^2K^{M^2},\;\\
R (e_1,e_3,e_2,e_3)=e_1(\kappa_2)-e_3(\sigma)-\kappa_{1}f_{1}-\kappa_1\kappa_2=-a_{1}^{3}a_{2}^{3}K^{M^2},\;\\
R (e_1,e_2,e_1,e_2)=e_1(f_2)-e_2(f_1)-f_{1}^{2}-f_{2}^{2}+2f_{3}\sigma-3\sigma^2=(a_{3}^{3})^2K^{M^2},\\
R (e_1,e_2,e_2,e_3)=-e_2(\sigma)+2\kappa_2\sigma=a_{1}^{3}a_{3}^{3}K^{M^2},\\
R (e_2,e_3,e_1,e_3)=e_2(\kappa_{1})+e_3(\sigma)+\kappa_2 f_2-\kappa_1 \kappa_2=-a_{1}^{3}a_{2}^{3}K^{M^2},\\
R (e_2,e_3,e_2,e_3)=\sigma^{2}+e_2(\kappa_2)-\kappa_1f_2- \kappa_2^2=(a_{1}^{3})^2K^{M^2}.
\end{cases}
\end{equation}
Finally, the Gauss curvature of the base space is given by
\begin{equation}\label{GCB}
K^N=e_1(f_2)-e_2(f_1)-f_1^2-f_2^2+2f_{3}\sigma.
\end{equation}
\begin{remark}
From (\ref{R2}) we see that $\sigma=0$ on  the neighborhood $W$ if and only if the horizontal distribution of the Riemannian submersion $\pi$ on $W$ is integrable.
\end{remark}

\subsection{ Harmonic  Riemannian submersions from  $M^2\times\r$}

Harmonic Riemannian submersions $M^3\to N^2$ are a special subclass of horizontally homothetic harmonic morphisms with totally geodesic fibers. For general harmonic morphisms and their applications and interesting links to other areas of mathematics see the book \cite{BW1}.  It is well known that the only Riemannian submersion from $S^3$ is the Hopf fibration $S^3\to S^2$ which is harmonic since the fibers are totally geodesic. Using the Bernstein theorem for harmonic morphism (see Baird-Wood \cite{BW2}), one can easily deduce (see also \cite{WO}) that there is no harmonic Riemannian submersion  $H^3\to (N^2, h)$ from a 3-dimensional hyperbolic space form, and that any globally defined harmonic Riemannian submersion $\phi: \r^3 \to (N^2, h)$ is an orthogonal projection $\r^3\to\r^2$ followed by a Riemannian covering map $\r^2\to (N^2,h)$. For some recent work on  the classifications of harmonic Riemannian submersions from Thurston's 3-dimensional geometries, BCV 3-spaces, and Berger 3-spheres see   \cite{WO4}.  In this subsection, we study harmonic Riemannian submersions from $M^2\times\r$.

\begin{theorem}\label{HTh}
A Riemannian submersion $\pi:M^2\times\r\to (N^2,h)$  is harmonic if and only if  locally it is, up to an isometry of the domain, the projection onto the first factor followed by a Riemannian covering map.
\end{theorem}

\begin{proof}
By Lemma \ref{2.1}, for any point on $M^2\times\r$, there is a neighborhood $W=U\times \r$ on which we have an orthonormal frame $\{E_i\}$. Let $e_i=\sum\limits_{j=1}^3a_{i}^{j}E_j$ be a local orthonormal frame with $e_3$ vertical.
Using Proposition 2.2 in \cite{WO4} and (\ref{RC0}),  we conclude that if the  Riemannian submersion $\pi:M^2\times\r\supseteq W\to (N^2,h)$ is harmonic, then we have
\begin{equation}\label{RCH}
\sigma^2=(a_{2}^{3})^2K^{M^2}=(a_{1}^{3})^2K^{M^2},\;
a_{1}^{3}a_{2}^{3}K^{M^2}=0,\;{\rm and}\;
K^N=3\sigma^2+(a_{3}^{3})^2K^{M^2}.
\end{equation}
By continuity of $K^{M^2}$, we may assume that either $K^{M^2}\equiv 0$ or $K^{M^2}\ne 0$ on a neighborhood denoted by $W$ by an abuse of notation.  For the first case, (\ref{RCH}) implies that $\sigma=0$, and $K^{N}=K^{M^2}=0$ on $W$. For the second case, we use $K^{M^2}\ne 0$ on $W$ together with the 1st and the 2nd equation of (\ref{RCH}) to conclude that $a_{1}^{3}=a_{2}^{3}=0$ and hence $(a_3^3)^2=1$, which also implies that  $\sigma=0$ and hence the Gauss curvature of the base space  $K^N=\sigma^2+(a_{3}^{3})^2K^{M^2}=(a_{3}^{3})^2K^{M^2}=K^{M^2}$. Thus, in either case, we have  $\sigma =0$ on a neighborhood $W$ which means the Riemannian submersion $\pi:M^2\times\r\supseteq W\to (N^2,h)$  has integrable horizontal distribution.  This, together with  the assumption that $\pi |_W$ is harmonic and hence it has totally geodesic fibers, implies that  $\pi |_{W}$ is a Riemannian submersion  with totally geodesic fibers and  integrable horizontal distribution whose integral submanifold is an open neighborhood $U$ of $M^2$ that is isometric to an open neighborhood of $(N^2, h)$. Therefore,  $\pi |_W$  is the projection along the fibers onto $U$ followed by a Riemannian covering map onto $V\subset N^2$.
\end{proof}

\subsection{ Biharmonic  Riemannian submersions from  $M^2\times\r$}

In this subsection, we will characterize all proper biharmonic  Riemannian submersions from $M^2\times\r$, which are not harmonic.\\
We will use the following lemma  in the rest of the paper.
\begin{lemma}(\cite{WO})\label{Lem1}
Let $\pi:(M^3,g)\to (N^2,h)$ be a Riemannian submersion
with an adapted frame $\{e_1,\; e_2, \;e_3\}$ and the integrability
data $ f_1, f_2, \kappa_1,\;\kappa_2\;{\rm and}\; \sigma$. Then, the
Riemannian submersion $\pi$ is biharmonic if and only if
\begin{equation}\label{lem1}
\begin{cases}
-\Delta^{M}\kappa_1-2\sum\limits_{i=1}^{2}f_i e_i(\kappa_2)-\kappa_2\sum\limits_{i=1}^{2}\left(e_i( f_i)
-\kappa_i f_i\right)+\kappa_1\left(-K^{N}+\sum\limits_{i=1}^{2}f_{i}^{2}\right)
=0,\\
-\Delta^{M}\kappa_2+2\sum\limits_{i=1}^{2}f_i e_i(\kappa_1)+\kappa_1\sum\limits_{i=1}^{2}(e_i( f_i)
-\kappa_i f_i)+\kappa_2\left(-K^{N}+\sum\limits_{i=1}^{2}f_{i}^{2}\right)=0,
\end{cases}
\end{equation}
where
$K^{N}=R^{N}_{1212}\circ\pi=e_1(f_2)-e_2(f_1)-f_{1}^{2}-f_{2}^{2}$\; is  the Gauss curvature of $(N^2,h)$.
\end{lemma}

\begin{lemma}(see \cite{WO1}) \label{L1}
Let $\pi:(M^3,g)\to (N^2,h)$ be a Riemannian submersion
from  Riemannian 3-manifolds  and  $\{e_1, e_2, e_3\}$ be  any local orthonormal frame with
$ e_3$ tangent to the fibers. If $\nabla_{e_{1}}e_{1}=0$, then either $\nabla_{e_{2}}e_{2}=0$;\;or\;$\nabla_{e_{2}}e_{2}\not\equiv0$, and  the
frame $\{e_1, e_2, e_3\}$  is  adapted  to the Riemannian submersion $\pi$.
\end{lemma}

The main tool used to prove our main theorems is a special orthonormal frame adapted to the Riemannian submersion. First we prove the existence of such frame.

\begin{proposition}\label{Clac}
Let $\pi: M^2\times\r\to (N^2,h)$ be a Riemannian submersion and $\{E_1,\; E_2,\; E_3=\frac{\partial}{\partial z}\}$ be a local semi-geodesic
  orthonormal  frame stated in Lemma \ref{2.1}.  Then,
there exists an orthonormal   frame \begin{equation}\label{GF}
\begin{cases}
 e_1=\cos\theta E_1+\sin\theta E_2, \\
 e_2=-\cos\alpha(\sin\theta E_1-\cos\theta E_2)+\sin\alpha E_3, \\
 e_3=\sin\alpha(\sin\theta E_1-\cos\theta E_2)+\cos\alpha E_3,
\end{cases}
\end{equation} such that  $e_3$ is vertical, $\nabla_{e_1}e_1=0$,  where $\cos\theta= \langle e_1, E_1\rangle,\; \cos \alpha=\langle e_3, E_3\rangle$.
Furthermore, the generalized integrability data of  $\{e_i\}$ are given by
 \begin{equation}\label{ida1}
\begin{array}{lll}
f_1=0,\;f_2=-\bar{f}\cos^2\alpha-\sin\alpha\cos\alpha E_3(\theta),\;f_3=-E_3(\theta),\\
\sigma=-\bar{f}\sin\alpha\cos\alpha-\sin^2\alpha E_3(\theta),\;
\kappa_1=-\bar{f}\sin^2\alpha+\sin\alpha\cos\alpha E_3(\theta),\\
\kappa_2=\sin\alpha e_3(\cos\alpha)-\cos\alpha e_3(\sin\alpha)=-e_3(\alpha),
\end{array}
\end{equation}
 where
 \begin{equation}\label{ida2}
\begin{array}{lll}
\bar{f}=-\sin \theta E_1(\theta)+\cos \theta E_2(\theta)+f \sin \theta.
\end{array}
\end{equation}
 \end{proposition}
\begin{proof}
 Let $e_3$ be the unit vector field tangent to the fibers of $\pi$. Clearly, if  $ e_3=\pm E_3$,  then $e_1=E_2, e_2=-E_1, e_3=E_3$ is such an  orthonormal frame.  \\
Hereafter, we suppose that $ e_3\ne \pm E_3=\pm \frac{\partial}{\partial z}$.  In this case,  we take $e_1=\frac{e_3\times E_3}{|e_3\times E_3|}$ which  is  horizontal since $\langle e_1, e_3\rangle=0$, and obtain a natural orthonormal frame $\{e_1, \;e_2=e_3\times e_1,\;e_3\}$ on $M^2\times\r$ which can be expressed as

\begin{equation}\label{eE}
e_i=\sum_{j=1}^3a_i^jE_j, i=1, 2, 3, (a_i^j)\in SO(3).
\end{equation}

By the choices $e_i$,  we have
\begin{equation}\label{zb}
\begin{array}{lll}
a_1^3=\langle e_1, E_3\rangle=0,\;a_3^{3}\neq\pm1,\;({\rm and \;hence})\; a_2^{3}\neq0,
\end{array}
\end{equation}
and hence  $e_1=a_1^1E_1+a_1^2E_2$ with $(a_1^1)^2+(a_1^2)^2=1$.\\
Furthermore, we can check that
\begin{equation}\label{bb1}
 f_1=0,\;\nabla_{e_1}e_1=0.
\end{equation}
In fact, a straightforward computation using (\ref{R2}) gives
\begin{equation}\label{bb2}
\begin{array}{lll}
-f_1\sum\limits_{i=1}^{3}a_2^iE_i=-f_1e_{2}=\nabla_{e_{1}} e_{1}=\nabla_{e_{1}}(\sum\limits_{i=1}^{3}a_1^iE_i)
=\sum\limits_{i=1}^{3}e_1(a_1^i)E_i+\sum\limits_{i,j=1}^{3}a_1^ja_1^i\nabla_{E_j}E_i.
\end{array}
\end{equation}
Using (\ref{g1}), the fact that $a_1^3=0$, and comparing the coefficient of $E_3$  on both sides of (\ref{bb2}) yields $-f_1a_2^3=e_1(a_1^3)=0$. From this we obtain $f_1=0$ since $a_2^3\neq0$, and hence $\nabla_{e_1}e_1=0$.\\

It is not difficult to check that
\begin{align}\label{Eb}
\bar{E}_1=a_1^2E_1-a_1^1E_2,\; \bar{E}_2=a_1^1E_1+a_1^2E_2,\; \bar{E}_3=E_3
\end{align}
defines another  orthonormal  frame on $U\times\r\subseteq M^2\times\r$ which satisfies
\begin{equation}\label{GG1}
\begin{array}{lll}
[\bar{E}_1,\bar{E}_3]=\bar{f}_3\bar{E}_2,\;[\bar{E}_2,\bar{E}_3]=-\bar{f}_3\bar{E}_1,\;\;[\bar{E}_1,\bar{E}_2]=\bar{f}\bar{E}_{1},\\
\nabla_{\bar{E}_{1}}\bar{E}_{1}=-\bar{f}\bar{E}_{2},\;\nabla_{\bar{E}_{1}}\bar{E}_{2}=\bar{f}\bar{E}_{1},\;\nabla_{\bar{E}_{3}}\bar{E}_{1}=-\bar{f}_3\bar{E}_2,
\;\nabla_{\bar{E}_{3}}\bar{E}_{2}=\bar{f}_3\bar{E}_1,\\
{\rm all\;other}\;\nabla_{\bar{E}_{i}}\bar{E}_{j}=0,\;i,\;j=1,2,3,
\end{array}
\end{equation}
with
\begin{equation}\label{GG2}
\begin{array}{lll}
 \bar{f}=E_1(a_1^1)+E_2(a_1^2)+fa_1^2, \;\bar{f}_3=a_1^2E_3(a_1^1)-a_1^1E_3(a_1^2),
\end{array}
\end{equation}
 and the only possible
nonzero components of the (Ricci) curvatures given by
\begin{equation}\label{g2}\notag
\begin{array}{lll}
{\rm Ric}\, (\bar{E}_{1},\bar{E}_{1})={\rm Ric}(\bar{E}_{2},\bar{E}_{2})= R_{1212}=g(R(\bar{E}_{1},\bar{E}_{2})\bar{E}_{2},\bar{E}_{1})=-\bar{E}_2(\bar{f})-\bar{f}^2,
\end{array}
\end{equation}
which is the Gauss curvature of $M^2$ on $U$ denoted by $K^{M^2}=-\bar{E}_2(\bar{f})-\bar{f}^2$. \\

We can also check that by introducing the new variables $\theta, \alpha$ so that
\begin{align}
 a_1^1=\cos \theta, a_1^2=\sin \theta, a_2^3=\sin \alpha, \;\;a_3^3= \cos \alpha,
 \end{align}
we can rewrite (\ref{Eb}) and (\ref{eE}), respectively, as
\begin{align}
\bar{E}_1=\sin\theta E_1-\cos\theta E_2,\;\bar{E}_2=\cos\theta E_1+\sin\theta E_2, \;\;\bar{E}_3=E_3,
\end{align}

\begin{equation}\label{hhb1}
\begin{cases}
 e_1=\cos\theta E_1+\sin\theta E_2, \\
 e_2=-\cos\alpha(\sin\theta E_1-\cos\theta E_2)+\sin\alpha E_3, \\
 e_3=\sin\alpha(\sin\theta E_1-\cos\theta E_2)+\cos\alpha E_3,
\end{cases}
\end{equation}
where
\begin{equation}\label{hhb2}
\begin{cases}
 a_1^1=\cos\theta,\; a_1^2=\sin\theta,\; a_1^3=0,\\
a_2^1=-\cos\alpha\sin\theta,\;a_2^2=\cos\alpha\cos\theta,\;a_2^3=\sin\alpha, \\
 a_3^1=\sin\alpha\sin\theta,\;a_3^2=-\sin\alpha\cos\theta,\;a_3^3=\cos\alpha.
\end{cases}
\end{equation}

It is also clear that  the relationship between the orthonormal  frames $\{e_i\}$ and $\{\bar{E}_i\}$ is given by
 \begin{equation}\label{hb3}
 e_1=\bar{E}_2,\;
 e_2=-\cos\alpha \bar{E}_1+\sin\alpha E_3,\;
 e_3=\sin\alpha \bar{E}_1+\cos\alpha E_3.
\end{equation}


To compute the integrability data of the  frame $\{e_1,\;e_2,\;e_3\}$, we first use (\ref{g1}), (\ref{R2}), $a_1^3=f_1=0$, and a  computation  similar to  those used  to obtain (\ref{bb2}), to have
\begin{equation}\label{thb2}
\begin{array}{lll}
e_1(a_{2}^{3})=-\sigma a_{3}^{3},\;
e_1(a_{3}^{3})=\sigma a_{2}^{3},\;
e_2(a_{2}^{3})=0,\;
e_2(a_{3}^{3})=0,\\
e_3(a_{2}^{3})=-\kappa_2a_{3}^{3},\;
e_3(a_{3}^{3})=\kappa_2a_{2}^{3},\;
\kappa_1a_{3}^{3}=(\sigma-f_{3})a_{2}^{3},\;
f_2a_{2}^{3}=\sigma a_{3}^{3},\\
e_1(a_1^1)=-a^1_1a_1^2f,\;e_1(a_1^2)=(a^1_1)^2f.
\end{array}
\end{equation}

Using (\ref{g1}), (\ref{hb3}), (\ref{GG1}),  the 1st and  the 2nd equations of (\ref{thb2}) we have
\begin{equation}\label{hb6}
\begin{array}{lll}
f_3 e_2+\kappa_1e_3=[e_1,e_3]=[e_1,\sin\alpha \bar{E}_1+\cos\alpha E_3]\\
=(\sigma+\bar{f}\sin\alpha\cos\alpha-\cos^2\alpha E_3(\theta))e_2+(-\bar{f}\sin^2\alpha+\sin\alpha\cos\alpha E_3(\theta))e_3.
\end{array}
\end{equation}

Comparing coefficients on both sides of (\ref{hb6}) yields
\begin{equation}\label{hb7}
\begin{array}{lll}
f_3=\sigma+\bar{f}\sin\alpha\cos\alpha-\cos^2\alpha E_3(\theta),\;
\kappa_1=-\bar{f}\sin^2\alpha+\sin\alpha\cos\alpha E_3(\theta).
\end{array}
\end{equation}

Similarly, by computing
\begin{equation}\label{hb8}
\begin{array}{lll}
f_2e_2-2\sigma e_3=[e_1,e_2]=[e_1,-\cos\alpha \bar{E}_1+\sin\alpha E_3]\\
=(-\bar{f}\cos^2\alpha-\sin\alpha\cos\alpha E_3(\theta))e_2+(-\sigma+\bar{f}\sin\alpha\cos\alpha+\sin^2\alpha E_3(\theta))e_3,
\end{array}
\end{equation}
and comparing coefficients on  both sides of this equation, we have
\begin{equation}\label{hb9}
\begin{array}{lll}
\sigma=-\bar{f}\sin\alpha\cos\alpha-\sin^2\alpha E_3(\theta),\;f_2=-\bar{f}\cos^2\alpha-\sin\alpha\cos\alpha E_3(\theta).
\end{array}
\end{equation}
Using the 1st equation of (\ref{hb7}) and the 1st equation of (\ref{hb9}), we obtain
\begin{equation}\label{hb10}
\begin{array}{lll}
f_3=- E_3(\theta)=\sin\theta E_3(\cos\theta)-\cos\theta E_3(\sin\theta)=a_1^2E_3(a_1^1)-a_1^1E_3(a_1^2).
\end{array}
\end{equation}
Using the 5th and 6th equation of (\ref{thb2}), a simple computation yields
\begin{equation}\label{k2}
\begin{array}{lll}
\kappa_2=\sin\alpha e_3(\cos\alpha)-\cos\alpha e_3(\sin\alpha)=-e_3(\alpha).
\end{array}
\end{equation}
From (\ref{bb1}),  (\ref{hb7}), (\ref{hb9}), (\ref{hb10}), and (\ref{k2}), we obtain  the integrability data (\ref{ida1}).
\end{proof}

Now we are ready to prove the following theorem which provides the main tool  to prove our classification theorems.

\begin{theorem}\label{Cla}
Let $\pi: M^2\times\r\to (N^2,h)$ be a Riemannian submersion.  Then,
the orthonormal   frame  $\{e_i\}$ given by (\ref{GF})  is adapted to the Riemannian submersion $\pi$ with the integrability data
 \begin{align}\label{ID0}
f_1=0,\;f_2=-\bar{f}\cos^2\alpha,\;f_3=0,\;
\sigma=-\bar{f}\sin\alpha\cos\alpha,\;
\kappa_1=-\bar{f}\sin^2\alpha,\;
\kappa_2=-e_3(\alpha).
\end{align}

\end{theorem}
\begin{proof}

By Proposition \ref{Clac}, we have  $\nabla_{e_{1}}e_{1}=0$. Thus, we can use Lemma \ref{L1} to conclude that either $\nabla_{e_{2}}e_{2}\not\equiv0$ in which case  the
chosen frame $\{e_1, e_2, e_3\}$ is adapted  to the Riemannian submersion $\pi$, or \;$\nabla_{e_{2}}e_{2}=0$. So, it suffices to prove that the orthonormal frame is also adapted in the the latter case: $\nabla_{e_{2}}e_{2}=0$, i.e., $f_2=0$.
Using the assumptions (\ref{zb}), $f_2=0$,  the 2nd and the 8th equations of (\ref{thb2})  we obtain $\sigma=0$.
Then, using the 4th equation of (\ref{RC0}) with $f_1=f_2=\sigma=a_1^3=0$,  we have either $K^{M^2}=0$, or $a_3^3=0$ and hence $(a_2^3)^2=1$.\\

For the case of $a_3^3=0$  and  $(a_2^3)^2=1$, one  applies  the 7th equation of  (\ref{thb2}) with $\sigma=0$ to conclude that $f_3=0$, which implies that the chosen frame $\{e_1=a_1^1E_1+a_1^2E_2, e_2, e_3\}$  is adapted to the Riemannian submersion.\\

For the case $K^{M^2}=0$ and $a_3^3\neq 0, \pm1$ (and hence $a_2^3\neq0,\pm1$),  we will prove that $f_3=0$ and hence the orthonormal frame $\{e_i\}$ is adapted to $\pi$.

First, we note that the orthonormal frame $\{\bar{E}_i\}$ given in (\ref{Eb}) is a natural orthonormal frame with respect to the harmonic Riemannian submersion
$$\pi_1:M^2\times\r\to M^2,\;\pi_1(p,z)=p.$$
By comparing the equations on the first line of (\ref{GG1}) with those of (\ref{R2})  and using  (\ref{hb10}) and the second equation of (\ref{GG2}), we find the generalized integrability data of $\{\bar{E}_i\}$ to be

\begin{align}\label{hhb3}
  \bar{f}_1 &=\bar{f},\;\bar{f}_2=\bar{\kappa}_1=\bar{\kappa}_2=\bar{\sigma}=0,\;
 \bar{f_3}=f_3=-E_3(\theta).
\end{align}

Applying Lemma \ref{L1} with $e_1=\bar{E}_2, e_2=\bar{E}_1$, we deduce that either $\bar{f}_1=\bar{f}\not\equiv0$ in which case the orthonormal frame $\{\bar{E}_i\}$ is adapted  to the Riemannian submersion $\pi_1$, and hence  $\bar{f_3}=f_3=0$, or $\bar{f}_1=\bar{f}=0$. For the latter case, we use the assumptions $\bar{f}_1=\bar{f}=0$, together with $a_3^3\neq 0, \pm1$ (and hence $a_2^3\neq0,\pm1$), $\sigma=f_2=0$ and (\ref{hb9}) to have $f_3=-E_3(\theta)=0$. Thus, in any case, the frame $\{e_i\}$ is adapted to the Riemannian submersion $\pi$, and the generalized integrability data (\ref{ida1}) reduces to the integrability data (\ref{ID0}).
\end{proof}

\begin{corollary}\label{Co1}
A biharmonic Riemannian submersion  $\pi:M^2\times\r\to (N^2,h)$  with $K^{M^2}=K^N=0$ has to be harmonic.
\end{corollary}
\begin{proof}
By Theorem \ref{Cla},  the orthonormal frame $\{e_i\}$ is adapted to the Riemannian submersion $\pi$ with the integrability data (\ref{ID0}).  Using the 4th equation of (\ref{RC0}) and the assumption that $K^{M^2}=K^N=0$, we have $\sigma=0$. From this and the 4th equation of (\ref{ID0}), we have either $\bar{f}=0$\; or \;$\sin\alpha\cos\alpha=0$. Therefore, we can obtain our corollary by the following cases:\\

Case I: $\bar{f}=0$ and $a_2^3a_3^3=\sin\alpha\cos\alpha\neq0$.
Combining these, (\ref{ID0}), and the assumption that $K^{M^2}=K^N=0$  we have
\begin{equation}\label{hhb11}
\begin{array}{lll}
a_1^1=f_1=f_2=f_3=\sigma=\kappa_1=K^{M^2}=K^N=0,\;a_2^3,\;a_3^3\neq0,\pm1,\;
\kappa_2=-e_3(\alpha).
\end{array}
\end{equation}
Using (\ref{hhb11}) and  (\ref{RC0}) we obtain
 \begin{equation}\label{hhb12}
\begin{array}{lll}
e_2(\kappa_2)=\kappa_2^2,\;e_1(\kappa_2)=0.
\end{array}
\end{equation}
Thus, in this case, if $\pi$ is biharmonic with $K^N=0$, then we have (\ref{R2}), (\ref{hhb11}), (\ref{hhb12}), and  the biharmonic equation (\ref{lem1}) reduces to
 \begin{equation}\label{hhb13}
\begin{array}{lll}
\Delta\kappa_2=0,
\end{array}
\end{equation}
and hence
 \begin{equation}\label{hhb14}
\begin{array}{lll}
e_3e_3(\kappa_2)=-\kappa_2^3.
\end{array}
\end{equation}
 Applying $e_2$ to both sides of (\ref{hhb14}) we have
  \begin{equation}\label{hhb15}
\begin{array}{lll}
e_2e_3e_3(\kappa_2)=-3\kappa_2^4.
\end{array}
\end{equation}
 On the other hand, we use (\ref{R2}) and $e_2e_3(\kappa_2)-e_3e_2(\kappa_2)=[e_2,e_3](\kappa_2)=\kappa_2e_3(\kappa_2)$ to obtain
 \begin{equation}\label{hhb16}
\begin{array}{lll}
e_2e_3(\kappa_2)=3\kappa_2e_3(\kappa_2).
\end{array}
\end{equation}
 By applying $e_3$ to both sides of (\ref{hhb16}) and using (\ref{hhb14}) we get
  \begin{equation}\label{hhb17}
\begin{array}{lll}
e_3e_2e_3(\kappa_2)=3(e_3(\kappa_2))^2-3\kappa_2^4.
\end{array}
\end{equation}
 A further computation using (\ref{hhb15}) and (\ref{hhb17}) gives
  \begin{equation}\label{hhb18}
\begin{array}{lll}
\kappa_2e_3(e_3(\kappa_2))=[e_2,e_3](e_3(\kappa_2))\\
=e_2e_3e_3(\kappa_2)-e_3e_2e_3(\kappa_2)
=-3(e_3(\kappa_2))^2.
\end{array}
\end{equation}
 This, together with (\ref{hhb14}), implies
 \begin{equation}\label{hhb19}
\begin{array}{lll}
3(e_3(\kappa_2))^2=\kappa_2^4,
\end{array}
\end{equation}
  Applying $e_3$ to both sides of (\ref{hhb19}) and using (\ref{hhb14})  we get
 \begin{equation}\label{hhb20}
\begin{array}{lll}
10\kappa_2^3e_3(\kappa_2)=0,
\end{array}
\end{equation}
 which implies either $\kappa_2=0$ or $e_3(\kappa_2)=0$. For the latter case, we use (\ref{hhb19}) to obtain $\kappa_2=0$. So  in any case,  the
Riemannian submersion $\pi$ is  harmonic since $\kappa_1=\kappa_2=0$.\\

Case II: $a_3^3=\cos\alpha=0$ and $ a_2^3=\sin\alpha=\pm1$. In this case, (\ref{ID0}) reduces to
 \begin{equation}\label{hhhb11}
\begin{array}{lll}
f_1=f_2=f_3=\sigma=\kappa_2=0,\;
\kappa_1=-\bar{f}.
\end{array}
\end{equation}
Using these and (\ref{RC0}) we have
 \begin{equation}\label{hhhb12}
\begin{array}{lll}
e_1(\kappa_1)=\kappa_1^2,\;e_2(\kappa_1)=0.
\end{array}
\end{equation}
It follows that in this case, we can use (\ref{R2}), (\ref{hhhb11}), (\ref{hhhb12}) and $K^N=0$ to conclude that the biharmonic equation (\ref{lem1}) reduces to
 \begin{equation}\label{hhhb13}
\begin{array}{lll}
\Delta\kappa_1=0.
\end{array}
\end{equation}
We make a computation similar to those used to compute  (\ref{hhb14})-(\ref{hhb20}) in Case I to have $\kappa_1=0$. So $\pi$ is harmonic since $\kappa_1=\kappa_2=0$.\\

Case III: $a_3^3=\cos\alpha=\pm1$ and $\sin\alpha=0$. In this case, using the 5th and 6th equations of
(\ref{ID0}) we have $\kappa_1=\kappa_2=0$ implying that the Riemannian submersion $\pi$ is  harmonic.\\

Summarizing all the above cases we obtain the corollary.

\end{proof}

Now we are ready to give a first  characterization of  proper biharmonic Riemannian submersions from $M^2\times\r$.
\begin{theorem}\label{Cla1}
Let $\pi:M^2\times\r\to (N^2,h)$ be a  proper biharmonic Riemannian submersion.  Then we have \\
$(a)$ The target surface is flat, and the adapted frame (\ref{GF}) has the integrability data $f_1=f_2=\kappa_2=\sigma=0,\; \kappa_1=-\bar{f} \ne 0$ solving the following PDE
\begin{equation}\label{zr1}
\begin{array}{lll}
\Delta\kappa_1=0,\;{\rm i.e.},\; \Delta^{M^2}\bar{f}=0,
\end{array}
\end{equation}
where $\Delta^{M^2}$ denotes the Laplacian of $M^2$ and $\bar{f}=-\sin \theta E_1(\theta)+\cos\theta E_2(\theta)+ f\sin \theta$ is a function on $U\subseteq M^2$,\; or,\\
$(b)$  The target surface is  non-flat,  and the adapted frame (\ref{GF}) has the integrability data
\begin{align}\label{B0}
f_1=\kappa_2=0,\; f_2=-\bar{f}\cos^2\alpha, \;\kappa_1=-\bar{f}\sin^2\alpha,\; \sigma=-\bar{f}\sin\alpha\cos\alpha
\end{align}
satisfying $f_2\kappa_1\sigma\neq0$, $K^N=e_1(f_2)-f_2^2\neq0$, $\sigma=-e_1(\alpha)$,\;$e_i(f_2)=e_i(\kappa_1)=e_i(\sigma)=e_i(K^{M^2})=e_i(\bar{f})=e_i(\alpha)=0,\;i=2,3$,\; and $\kappa_1$ solving the following PDE
\begin{equation}\label{zr2}
\begin{array}{lll}
\Delta\kappa_1-\kappa_1\{-K^N+f_2^2\}=0,
\end{array}
\end{equation}
which is equivalent to
\begin{equation}\label{zr3}
\begin{array}{lll}
\alpha'''\sin\alpha\cos^2\alpha+\cos\alpha(\sin^2\alpha+3)\alpha'\alpha''+\sin\alpha(2\cos^2\alpha+3)\alpha'^3=0,
\end{array}
\end{equation}
where $\alpha',\;\alpha''$\;and $\alpha'''$ denote the first, the second and the third derivative of $\alpha$ along the vector field $e_1$.
\end{theorem}
\begin{proof}
By  Theorem \ref{Cla}, the local orthonormal   frame  given by (\ref{GF}) is  adapted to the Riemannian submersion $\pi$  with $e_3$ being vertical and the integrability data given by
\begin{align}\label{cI}
\begin{cases}
f_1=0,\;f_2=-\bar{f}\cos^2\alpha,\;f_3=-E_3(\theta)=0,\; \sigma=-\bar{f}\sin\alpha\cos\alpha,\\
\kappa_1=-\bar{f}\sin^2\alpha,\;\kappa_2=\sin\alpha e_3(\cos\alpha)-\cos\alpha e_3(\sin\alpha)=-e_3(\alpha).
\end{cases}
\end{align}

 We may assume that  $\cos\alpha=a_3^3\ne \pm1$, for otherwise $e_3=\pm E_3$, $\pi$ is the projection $M^2\times\r\to M^2$ followed by a Riemannian covering map and hence $\pi$ is harmonic. The rest of the proof will be done by considering the following  three cases.

Case I: $a_3^3=\cos\alpha=0$. In this case,  it follows from (\ref{cI}) that $f_1=f_2=\kappa_2=\sigma=0,\; \kappa_1=-\bar{f}\neq0$. From these and  (\ref{GCB}), we  conclude that $N^2$ has Gauss curvature  $K^N=0$ and hence it is a flat surface.  Therefore, biharmonic equation (\ref{lem1}) turns into
 $$\Delta\kappa_1=0.$$
 It is easy to see that the above equation is equivalent to $\Delta^{M^2}\bar{f}=0$ since $E_3\theta=0$ and hence $E_3\bar{f}=0$ and $\Delta {\bar f}=\Delta^{M^2}  {\bar f}+E_3E_3 {\bar f}=\Delta^{M^2}  {\bar f}$.\\

Case II: $a_3^3=\cos\alpha\neq0,\pm1$ and $f_2=0$. In this case, the second equation of (\ref{cI}) implies that $\bar{f}=0$, and hence we have $f_1=f_2=f_3=\kappa_1=\sigma=K^{M^2}=K^N=0$. Thus, we apply Corollary \ref{Co1} to conclude that the biharmonic Riemannian submersion is harmonic in this case.

Case III: $a_3^3=\cos\alpha\neq0,\pm1$ and   $f_2\neq0$.  In this case,  the hypotheses can be summarized as:
\begin{equation}\label{z1}
a_2^3,\;a_3^3\neq0, \pm1,\;f_2\neq0,\;a_1^3=f_1=f_3=0, \;e_3(f_1)=e_3(f_2)=0.
\end{equation}
{\bf Claim 1}: In this case, we have $\kappa_2=0$, $e_2(\bar{f})=e_3(\bar{f})=e_2(f_2)=e_3(f_2)=e_2(\kappa_1)=e_3(\kappa_1)=e_2(\sigma)=e_3(\sigma)=e_2(K^{M^2})=e_3(K^{M^2})=e_2(\alpha)=e_3(\alpha)=0,\;\kappa_1\sigma\neq0$, and $\alpha'=-\sigma=\bar{f}\sin\alpha\cos\alpha $, where $\alpha'$ denotes the first  derivative of $\alpha$ along $e_1=\bar{E}_2$.\\
 {\bf Proof of Claim 1}: Firstly, we use (\ref{z1}), the 7th and the 8th equation of (\ref{thb2}) to have $\kappa_1\sigma\neq0$.\\
  Secondly, by using (\ref{cI}) and the last equation of (\ref{z1}), we have $e_3(-\cos^2\alpha \bar{f})=e_3(f_2)=0$ and hence
\begin{equation}\label{bs4}
\begin{array}{lll}
e_3(\kappa_1)=e_3(-\sin^2\alpha \bar{f})=e_3[-(1-\cos^2\alpha) \bar{f}]=-e_3(\bar{f}).
\end{array}
\end{equation}
We use (\ref{cI}), the 5th and the 7th equation of (\ref{RC0}) with $a_1^3=0$ to obtain
\begin{equation}\label{bs5}
\begin{array}{lll}
\sigma^2=\kappa_1f_2,\;e_2(\kappa_2)=\kappa_2^2,\;e_2(\sigma)=2\kappa_2\sigma.
\end{array}
\end{equation}
On the other hand, note that $a_2^3=\sin\alpha,\;a_3^3=\cos\alpha$, a straightforward computation using the 6th equation of  (\ref{thb2}) gives
\begin{equation}\notag
\begin{array}{lll}
0=e_3(f_2)=e_3(-\cos^2\alpha \bar{f})
=-2\kappa_2\sin\alpha\cos\alpha \bar{f}-\cos^2\alpha e_3(\bar{f}).
\end{array}
\end{equation}
It follows that
\begin{equation}\label{bs6}
\begin{array}{lll}
 e_3(\bar{f})=\frac{-2\kappa_2\sin\alpha\cos\alpha \bar{f}}{\cos^2\alpha}=\frac{2\kappa_2\sigma}{\cos^2\alpha}.
\end{array}
\end{equation}
Using  (\ref{bs4}), (\ref{bs6}), the 1st equation of (\ref{bs5}) and the 3rd equation of  (\ref{RC0}) with $a_1^3=f_1=0$, we obtain
\begin{equation}\label{bs7}
\begin{array}{lll}
e_3(\sigma)=\frac{e_3(\sigma^2)}{2\sigma}=\frac{f_2e_3(\kappa_1)}{2\sigma}=\kappa_2\bar{f},\\
e_1(\kappa_2)=e_3(\sigma)+\kappa_1\kappa_2=\kappa_2\bar{f}-\kappa_2\sin^2\alpha \bar{f}=\kappa_2\cos^2\alpha \bar{f}=-\kappa_2f_2.
\end{array}
\end{equation}

On the other hand, using $e_2(\cos\alpha)=e_2(\sin\alpha)=0$, $\bar{f}=-\frac{\sigma}{\sin\alpha\cos\alpha}$  and $e_2(\sigma)=2\kappa_2\sigma$ and a simple computation we have
\begin{equation}\label{bs8}
\begin{array}{lll}
e_2( \bar{f})=-\frac{e_2(\sigma)}{\sin\alpha\cos\alpha}=\frac{-2\kappa_2\sigma}{\sin\alpha\cos\alpha}=2\kappa_2 \bar{f},\\
e_2(f_2)=e_2(-\cos^2\alpha \bar{f})=-\cos^2\alpha e_2( \bar{f})=-2\kappa_2\cos^2\alpha \bar{f}=2\kappa_2f_2,\\
e_2(\kappa_1)=e_2(-\sin^2\alpha \bar{f})=-\sin^2\alpha e_2( \bar{f})=-2\kappa_2\sin^2\alpha \bar{f}=2\kappa_1\kappa_2.
\end{array}
\end{equation}
Note that $\kappa_2e_3(f_2)=[e_2,e_3](f_2)=e_2 e_3(f_2)-e_3 e_2(f_2)$,  a further computation using (\ref{bs8})
gives $0=-e_3 (2\kappa_2f_2)=-f_2e_3(\kappa_2)$, which, together with $f_2\neq0$, implies that
\begin{equation}\label{bs9}
\begin{array}{lll}
e_3(\kappa_2)=0.
\end{array}
\end{equation}
Applying the 2nd equation of biharmonic equations (\ref{lem1}) and using (\ref{R2}), (\ref{bs5}), (\ref{bs7}), (\ref{bs8}), (\ref{bs9}) and (\ref{GCB}) with $f_1=f_3=0$, we have
\begin{equation}\label{bs10}
\begin{array}{lll}
0=-\Delta\kappa_2+2f_2e_2(\kappa_1)+\kappa_1e_2(f_2)-\kappa_1\kappa_2f_2+\kappa_2\{-K^N+f_2^2\}\\
=\sum\limits_{i=1}^3\{-e_ie_i(\kappa_2)+\nabla_{e_i}e_i(\kappa_2)\}+4\kappa_1\kappa_2f_2+2\kappa_1\kappa_2f_2-\kappa_1\kappa_2f_2+\kappa_2\{-e_1(f_2)+2f_2^2\}\\
=-\kappa_2^3+4\kappa_1\kappa_2f_2=\kappa_2(-\kappa_2^2+4\kappa_1f_2)=\kappa_2(-\kappa_2^2+4\sigma^2).
\end{array}
\end{equation}
Solving (\ref{bs10}) we must have $\kappa_2=0$ and $\kappa_2^2\neq4\sigma^2$. Indeed, if $\kappa_2^2=4\sigma^2$, then
using (\ref{bs9}), we have $4e_3(\sigma^2)=e_3(\kappa_2^2)=0$ and hence $e_3(\sigma)=0$. But since  $\bar{f}\neq0$ and (\ref{bs7}), we further have  $\kappa_2=0$ and hence $4\sigma^2=\kappa_2^2=0$, a contradiction, since $\sigma\neq0$.
Therefore, a straightforward computation using (\ref{bs4})--(\ref{bs8}), the 5th, the 6th equation  equation of  (\ref{thb2}) gives
\begin{equation}\label{bs11}
e_2(\bar{f})=e_3(\bar{f})=e_3(\sin\alpha)=e_3(\cos\alpha)=e_2(\kappa_1)=e_3(\kappa_1)=e_3(\sigma)=e_2(\sigma)=e_2(f_2)=0.
\end{equation}
Since $K^{M^2}=-\bar{E}_2(\bar{f})-\bar{f}^2=-e_1(\bar{f})-\bar{f}^2$, we see that $e_3(K^{M^2})=e_2(K^{M^2})=0$. From (\ref{thb2}) and (\ref{bs11}), we have  $e_2(\alpha)=e_3(\alpha)=0$, and a direct computation gives $\sigma\sin\alpha =e_1(\cos\alpha)=-\sin\alpha \alpha'$, i.e., $\alpha'=-\sigma=\sin\alpha\cos\alpha \bar{f}$, where $\alpha'$ denotes the first  derivative of $\alpha$ along $e_1=\bar{E}_2$.  For $\kappa_2=f_1=0$, biharmonic equation (\ref{lem1}) turns into (\ref{zr2}). A further computation using Claim 1 and (\ref{cI}), we see that (\ref{zr2}) turns into (\ref{zr3}).\\

 Finally, we show that $(N^2, h)$ is not flat. If otherwise, $K^N=0$, i.e., $e_1(f_2)-f_2^2=0$, using  the 4th equation of (\ref{RC0}) with $\sigma =-\bar{f}\sin\alpha \cos\alpha$, we immediately have
 \begin{equation}\label{bs12}
K^{M^2}=-3\bar{f}^2\sin^2\alpha.
\end{equation}
On the other hand, we know that  $M^2$ has Gauss curvature
\begin{equation}\label{bs13}
K^{M^2}=-\bar{E}_2(\bar{f})-\bar{f}^2=-e_1(\bar{f})-\bar{f}^2.
\end{equation}
Using Equations (\ref{bs12}) and (\ref{bs13}), we obtain
\begin{equation}\label{bs14}
\bar{f}'=3\bar{f}^2\sin^2\alpha-\bar{f}^2,
\end{equation}
  where $\bar{f}'$ denotes the first  derivative of $\bar{f}$ along $e_1$. \\
Substituting  $\alpha'=\sin\alpha\cos\alpha \bar{f}$ (i.e., $\bar{f}=\frac{\alpha'}{\sin\alpha\cos\alpha}$) into (\ref{bs14}) and simplifying the resulting equation we get
\begin{equation}\label{bs15}
\alpha''\cos\alpha-\alpha'^2\sin\alpha=0.
\end{equation}
One applies $e_1$ to both sides of (\ref{bs15}) and simplifies the resulting equation to have
\begin{equation}\label{bs16}
\alpha'''\cos\alpha-3\alpha'\alpha''\sin\alpha-\alpha'^3\cos\alpha=0.
\end{equation}
Adding (\ref{zr3}) to a $(-\sin\alpha\cos\alpha)$ multiple of (\ref{bs16}) and simplifying the results  with $\alpha'\neq0$ yields
\begin{equation}\label{bs17}
\alpha''\cos\alpha(4\sin^2\alpha+3)+\alpha'^2\sin\alpha(3\cos^2\alpha+3)=0.
\end{equation}
Similarly, adding  a $(-4\sin^2\alpha-3)$ multiple of (\ref{bs15}) to  (\ref{bs17}) and simplifying the results gives
\begin{equation}\label{bs17}
\alpha'^2\sin\alpha(9+\sin^2\alpha)=0,
\end{equation}
it follows that $\alpha$ is a constant, a contradiction. Then, we must have $K^N\neq0$.\\

Summarizing all results in the above cases we obtain the theorem.
\end{proof}

We now give a characterization of  proper biharmonic Riemannian submersions from $M^2\times\r$ by using the local coordinates as follows

\begin{theorem}\label{Ths1}
If $\pi:M^2\times\r\to (N^2,h)$ is a  proper biharmonic Riemannian submersion from the  product space, then\\
(i) The target surface is flat, and  locally, up to an isometry of the domain and/or codomain, $\pi$ is the projection of the special twisted product
\begin{equation}\label{TP}
 \pi: (\r^3, e^{2p(x, y)}dx^2+dy^2+dz^2) \to(\r^2,dy^2+dz^2),\;\pi(x, y, z)=(y,z),
 \end{equation}
with  $p_y\neq0$ being a harmonic function on  $(M^2,e^{2p(x, y)}dx^2+dy^2)$, i.e., it solves the PDE \\
\begin{equation}\label{mr}
\begin{array}{lll}
\Delta p_y :=p_{yyy}+p_{yy}p_y+e^{-2p(x, y)}(p_{xxy}-p_{xy}p_x)=0.
\end{array}
\end{equation}
or,\\
(ii) The target surface is non-flat, and  locally, up to an isometry of the domain and/or codomain, the map can be  expressed as
\begin{equation}\label{th00}
\begin{array}{lll}
\pi:(\r^3,e^{2p(x,y)}dx^2+dy^2+dz^2)\to (\r^2, dy^2+e^{2\lambda(y,\phi)}d\phi^2),\\
\pi(x, y, z)=(y,F\left(z- \int e^{\varphi(x)}dx\right)),
\end{array}
\end{equation}
where $p(x,y)=\ln|\tan\alpha(y)|+\varphi(x)$, $\lambda=\ln |\sin \alpha (y)|+w(\phi)$ with the functions $\varphi(x)$, $w(\phi)$  and nonconstant function $F(u)$ satisfying  $F'(z- \int e^{\varphi(x)}dx)=e^{-w(\phi)}$ and $z- \int e^{\varphi(x)}dx=\int e^{w(\phi)}d\phi$,  and $\alpha(y)$ is the angle between the fibers of $\pi$ and $ E_3=\partial_z$ solving the ODE (\ref{zr3}).
\end{theorem}
\begin{proof}

First of all, note that by Theorem \ref{Cla1}  the local orthonormal  frame  $\{e_1=\bar{E}_2,\; e_2=-\cos\alpha \bar{E}_1+\sin\alpha \bar{E}_3,
\;e_3=\sin\alpha \bar{E}_1+\cos\alpha \bar{E}_3\}$ is an adapted frame of the Riemannian submersion $\pi$, and that the vector field $e_1=\bar{E}_2$ is a geodesic vector field on $M^2$. It is well known that we can choose local semi geodesic coordinates $(x,y)$ on  $M^2$  so that  $\bar{E}_1=e^{-p(x,y)}\partial_x, \bar{E}_2=\partial_y$, and  the metric on  $M^2$  takes the form $e^{2p(x,y)}dx^2+dy^2$. It follows that the product manifold $M^2\times\r$ can be locally represented as $(U\times\r\subseteq M^2\times\r, e^{2p(x,y)}dx^2+dy^2+dz^2)$.

For Statement (i), since the target surface is flat, it corresponds to the Case I in the proof of  Theorem \ref{Cla1}, i.e.,  the frame $\{e_1=\bar{E}_2,\; e_2= E_3,\;e_3=\bar{E}_1\}$ is an adapted frame  to the Riemannian submersion $\pi$ with the integrability data $\{f_1=f_2=\kappa_2=\sigma=0,\; \kappa_1=-\bar{f}\neq0\}$. Note that, by $\sigma=0$ and the  4th equation of (\ref{RC0}),  the horizontal distribution of the Riemannian submersion is integrable with flat integral submanifolds. So locally, up to an isometry of the domain  and/or the target manifold, the Riemannian submersion is the projection along the fibers (i.e., the integral curves of $\bar{E}_1=e^{-p(x,y)}\partial_x$ to the integral submanifold, and hence can be described by (\ref{TP})). It is easily checked that in this case $\kappa_1=-\bar{f}=-p_y$ and Equation (\ref{zr1}) reduces to (\ref{mr}).

For Statement (ii), Theorem \ref{Cla1} implies that  in this case, the target surface is  non-flat, $\cos\alpha\neq$ constant depending only on variable $y$, and $\{e_1=\bar{E}_2,\; e_2=-\cos\alpha \bar{E}_1+\sin\alpha \bar{E}_3,\;e_3=\sin\alpha \bar{E}_1+\cos\alpha \bar{E}_3\}$ is an adapted frame of the Riemannian submersion $\pi$ with integrability data (\ref{B0}).

Note that  the vector field $e_1=\bar{E}_2=\partial_y$ is a  basic vector field to the Riemannian submersion $\pi$. It is well known that there is a local vector field $\varepsilon_1$ on $(N^2, h)$ whose integral curves are geodesics on $(N, h)$ such that   $d \pi(e_1)=\varepsilon_1$.  It follows that we can choose a local semi geodesic coordinates $(y, \phi)$ on $N$ so that metric takes the form $h=dy^2+e^{2\lambda(y,\phi)}d\phi^2$, and the orthonormal frame  $\varepsilon_1=\partial_y, \varepsilon_2=e^{-\lambda}\partial_{\phi}$ satisfying  $d \pi(e_1)=\varepsilon_1, d \pi(e_2)=\varepsilon_2=e^{-\lambda}\partial_{\phi}$.

Summarizing the above, we conclude that in this case, up to and  isometry of the domain and/or target manifold, the Riemannian submersion can be expressed as
\begin{align}
\pi: (\r^3, e^{2p(x,y)}dx^2+dy^2+dz^2) \to(\r^2, h=dy^2+e^{2\lambda(y,\phi)}d\phi^2),\;\pi(x, y, z)=(y,\phi),
\end{align}
where $\phi=\phi(x, y, z)$ is a function to be determined.

Now we are to determine the functions $\phi=\phi(x,y, z)$, $p(x,y)$, and $\lambda(y, \phi)$.

A straightforward computation gives
\begin{equation}\label{d1}
\begin{array}{lll}
F_1\circ\pi\varepsilon_1+F_2\circ\pi\varepsilon_2=[\varepsilon_1,\varepsilon_2]=-\lambda_y\varepsilon_2,
\end{array}
\end{equation}
which implies
\begin{equation}\label{d2}\notag
\begin{array}{lll}
F_1=0,\;F_2=-\lambda_y.
\end{array}
\end{equation}
Hence, we have
\begin{equation}\label{d3}
\begin{array}{lll}
\lambda_y=-F_2\circ\pi=-f_2=\bar{f}\cos^2\alpha=\frac{\cos\alpha\alpha'(y)}{\sin\alpha}.
\end{array}
\end{equation}
Integrating both sides with respect to $y$ yields
\begin{equation}\label{d4}
\begin{array}{lll}
\lambda(y,\phi)=\int \frac{\cos\alpha\alpha'(y)}{\sin\alpha}dy+w(\phi)=\ln|\sin\alpha(y)|+w(\phi),
\end{array}
\end{equation}
where $w= w(\phi)$ is an arbitrary function on $\phi$.\\

 Noting that $p_y=\bar{f}(y)=\frac{\alpha'(y)}{\sin\alpha\cos\alpha}$ does not depend on $x$  we have
  \begin{equation}\label{d5}
\begin{array}{lll}
p(x,y)=\int \bar{f}(y)dy+\varphi(x)=\int\frac{\alpha'(y)}{\sin\alpha\cos\alpha}dy+\varphi(x)=\ln|\tan\alpha(y)|+\varphi(x),
\end{array}
\end{equation}
where $\varphi= \varphi(x)$ is a function on $x$.\\

To determine  the component function $\phi(x, y, z)$,  we use $e_1=\partial_y, e_2=-\cos\alpha e^{-p(x,y)}\partial_x+\sin\alpha \partial_z, e_3=\sin\alpha e^{-p(x,y)}\partial_x+\cos\alpha \partial_z$, and \;$d\pi(e_1)=\varepsilon_1=\partial_y,\; d\pi(e_2)=\varepsilon_2=e^{-\lambda}\partial_{\phi}$, and $d\pi(e_3)=0$ to have
\begin{equation}\label{c3}
\begin{array}{lll}
\partial_y=\varepsilon_1 &=d\pi(e_1)=\partial_y+\phi_y\,\partial_{\phi},\\
e^{-\lambda}\frac{\partial}{\partial \phi}=\varepsilon_2& =d\pi(e_2)=(-\cos\alpha e^{-p(x,y)}\phi_x\,+\sin\alpha \phi_z\,)\partial_{\phi},\\
0=& d\pi(e_3)=(\sin\alpha e^{-p(x,y)}\phi_x\,+\cos\alpha \phi_z\,)\partial_{\phi}.
\end{array}
\end{equation}
By comparing both sides of the 1st equation of (\ref{c3}), one finds that $\phi_y=\frac{\partial\phi}{\partial y}=0$,
which means that the function $\phi$ does not depend on $y$, i.e., $\phi=\phi(x,z)$.\\
Comparing coefficients of both sides of  the 2nd and the 3rd equation of (\ref{c3}) separately, we get
\begin{equation}\label{c4}
-\cos\alpha e^{-p(x,y)}\phi_x\,+\sin\alpha \phi_z=e^{-\lambda},\;
\sin\alpha e^{-p(x,y)}\phi_x\,+\cos\alpha \phi_z=0.
\end{equation}

Recall that $p(x,y)=\ln|\tan\alpha(y)|+\varphi(x)$, $\lambda=\ln |\sin \alpha (y)|+w(\phi)$, and the fact that  $\phi(x,z)$ is nonconstant since $d\pi(e_2)\neq0$, we use the method of the first integral to solve the 2nd  $PDE$ of (\ref{c4}) to have
\begin{equation}\label{c6}
\phi(x,z)= F\left(z-\int e^{\varphi(x)}dx\right),
\end{equation}
where $F=F(u)$ is a nonconstant differentiable function. Substituting this into the 1st  $PDE$ of (\ref{c4}) we have $F'(z-\int e^{\varphi(x)}dx)=e^{-w(\phi)}$ . It follows from this and (\ref{c6}) that $d\phi=d F=F'du=e^{-w(\phi)}du$ and hence $e^{w(\phi)}d\phi=du$ implying that $u=z-\int e^{\varphi(x)}dx=\int e^{w(\phi)}d\phi$.
This completes the proof of Statement (ii).
\end{proof}

 Applying Theorem \ref{Ths1}, we immediately have the following corollary which characterizes a proper biharmonic Riemannian submersions from a product manifold onto a non-flat surface  as a special map determined  up to an arbitrary function between two special warped product manifolds with the warping functions solving an ODE.
\begin{corollary}\label{co1}
A proper biharmonic Riemannian submersion $\pi:M^2\times\r\to (N^2,h)$ from product manifold into a non-flat surface is locally, up to an isometry of the domain and/or codomain, $\pi$ is a map between two special warped product spaces  given by
\begin{equation}\label{coo1}
\begin{array}{lll}
\pi:(\r^3, \tan ^2\alpha(y)\,d t^2+dy^2+dz^2)\to (\r^2,dy^2+\sin^2 \alpha(y)\,d\psi^2)\\
\pi(t,y,z)=(y,z- t),
\end{array}
\end{equation}
where $\alpha(y)$ is the angle between the fibers of  $\pi$ and $ E_3=\frac{\partial}{\partial z}$ solving the ODE
\begin{equation}\label{cx}
\begin{array}{lll}
\alpha'''\sin\alpha\cos^2\alpha+\cos\alpha(\sin^2\alpha+3)\alpha'\alpha''+\sin\alpha(2\cos^2\alpha+3)\alpha'^3=0,
\end{array}
\end{equation}
\end{corollary}
\begin{proof}
This  follows from Statement (ii) of Theorem \ref{Ths1} and   the coordinate changes  $t=\int{e^{\varphi(x)}}dx,\; y=y,\;z=z$ in the domain  and $ y=y, \psi=\int{e^{w(\phi)}}d\phi$ in the codomain.
\end{proof}

\begin{remark}\label{r3}
(A) Note that it follows form \cite{AO}  (Corollary 3.2) that the Riemannian submersion given by the projection of the twisted product $\pi: (\r^3, e^{2p(x, y,z)}dx^2+dy^2+dz^2) \to(\r^2,dy^2+dz^2),\;\pi(x, y, z)=(y,z)$ is biharmonic if and only if $\Delta \kappa_1=0, \Delta \kappa_2=0$. In the case of  (i) in Theorem \ref{Ths1}, these reduce exactly to (\ref{mr}). Thus, Statement (i) of Theorem \ref{Ths1} not just characterizes proper biharmonic Riemannian submersions from product manifold to a flat surface locally but also recovers a special result in \cite{AO}  (Corollary 3.2).\\
(B) The biharmonicity of the Riemannian submersion defined by the projection of warped product $\pi: (\r^3, e^{2p( y,z)}dx^2+dy^2+dz^2) \to(\r^2,dy^2+dz^2),\;\pi(x, y, z)=(y,z)$ had been studied in \cite{BMO, LO, WO, GO, AO}.
\end{remark}

(D) We would like to point out that, unlike in case (ii) in Theorem \ref{Ths1} and Corollary \ref{co1} where the function $p(x, y)$ can be proved to be independent of $x$ variable so the metric is of a warped product type, in case (i) the map is completely determined (to be the projection)  but we do not know whether the function $p(x, y)$ is independent of $x$ variable  or not.\\

When the target surface is flat,  it is easy to have many examples proper biharmonic Riemannian submersions $\pi:M^2\times\r\to \r^2$ from the projection of the warped product spaces see e.g., \cite{BMO, LO, GO, AO}. For example, the following projections are proper biharmonic Riemannian submersions:
\begin{align} \notag
&(i)\; \pi :(\r^{2}\times\r,(\cosh y)^4dx^2+dy^2+dz^2)
\to (\r^2 ,dy^2 + dz^2),\; \pi(x,y,z) =(y,z),\\\notag
&(ii)\;
\pi :(\r^2_{+}\times\r,y^4dx^2+dy^2+dz^2)
\to (\r^2_{+} ,dy^2 + dz^2),\; \pi(x,y,z) =(y,z).
\end{align}

When the target surface is non-flat, we would like to point out that there exist many local proper biharmonic Riemannian submersions  $\pi:M^2\times\r\to (N^2,h)$. In fact,  by introducing new variable $u(\alpha)=\frac{\alpha''(y)}{\alpha'(y)^2}$,  we have  $\alpha'\alpha''=u\, \alpha'^3$, $\alpha'''=(u'+2u^2) a'^3$ and hence
(\ref{cx}) reduces to a Riccati equation as
\begin{equation}\label{rcx}
\begin{array}{lll}
u'(\alpha)+2u^2+\frac{\sin^2\alpha+3}{\sin\alpha\cos\alpha}u+\frac{2\cos^2\alpha+3}{\cos^2\alpha}=0,
\end{array}
\end{equation}
any solution of which gives a family of locally defined proper biharmonic Riemannian submersions  $M^2\times\r\to (N^2,h)$.

Finally, note that it was proved in \cite{WO, WO1} that a proper biharmonic Riemannian submersion  $M^2(c)\times \r$ exists only in the case when $c<0$, and
$\pi:(\r^{3},e^{2\sqrt{-c}\,y}dx^2+dy^2+dz^2)\to (\r^2 ,dy^2 + dz^2),\; \pi(x,y,z) =(y,z)$ is an example. Now we can prove that,  up to isometry, this is the only one.   \begin{proposition}\label{Pr3}
A Riemannian submersion  $\pi:M^2(c)\times\r\to (N^2,h)$ is proper biharmonic if and only if $c<0$,  $(N^2,h)$ is flat, and, up to an isometry, the map can be expressed as $\pi: H^2(c)\times\r\to\r^2$ with
$\pi:(\r^{3},e^{2\sqrt{-c}\,y}dx^2+dy^2+dz^2)\to (\r^2 ,dy^2 + dz^2),\; \pi(x,y,z) =(y,z)$.
\end{proposition}
\begin{proof}
Firstly, it follows from \cite{WO1} that  a proper biharmonic Riemannian submersion
 $\pi:M^2(c)\times\r\to (N^2,h)$ from a product space exists only in the case: $H^2(c)\times\r\to\r^2$ with $c<0$.\\

 Secondly, by Theorem \ref{Cla1} and \ref{Ths1}, we know that
  locally, up to an isometry of the domain and/or codomain, a proper biharmonic Riemannian submersion
 $\pi:H^2(c)\times\r\to\r^2$ with $c<0$ is expressed as
\begin{equation}\label{th2}
 \pi: (\r^3, e^{2p(x, y)}dx^2+dy^2+dz^2) \to(\r^2,dy^2+dz^2),\;\pi(x, y, z)=(y,z),
 \end{equation}
 and  the orthonormal frame $\{e_1=\partial_y,\; e_2= \partial_z,\;e_3=e^{-p}\partial_x\}$ is adapted  to the Riemannian submersion $\pi$ with the integrability data $f_1=f_2=\kappa_2=\sigma=0,\; \kappa_1=-p_y \ne 0$.
It is easily checked that in this case, (\ref{RC0}) reduces to
\begin{equation}\label{th3}
e_1(\kappa_1)=\kappa_1^2+c,\;e_2(\kappa_1)=0,
\end{equation}
and biharmonic equation (\ref{lem1}) reads
\begin{equation}\label{th4}
\Delta\kappa_1=0.
\end{equation}
A straightforward computation gives
\begin{equation}\label{th5}
\begin{array}{lll}
\Delta\kappa_1=e_1e_1(\kappa_1)+e_3e_3(\kappa_1)-\nabla_{e_1}{e_1}(\kappa_1)-\nabla_{e_2}{e_2}(\kappa_1)-\nabla_{e_3}{e_3}(\kappa_1)\\
=e_1(\kappa_1^2+c)+e_3e_3(\kappa_1)-\kappa_1e_1(\kappa_1)=e_3e_3(\kappa_1)+\kappa_1^3+c\kappa_1.
\end{array}
\end{equation}
Substituting  (\ref{th5}) into (\ref{th4}), we have
\begin{equation}\label{th6}
\begin{array}{lll}
e_3e_3(\kappa_1)=-\kappa_1^3-c\kappa_1.
\end{array}
\end{equation}
Applying $e_3$ to both sides of  the 1st equation of (\ref{th3}) and using the fact that $e_1e_3(\kappa_1) = [e_1, e_3](\kappa_1) + e_3e_1(\kappa_1)$,
we have
\begin{equation}\label{th7}
\begin{array}{lll}
e_1e_3(\kappa_1)=3\kappa_1e_3(\kappa_1).
\end{array}
\end{equation}
Using  (\ref{th3}), (\ref{th6}),  (\ref{th7}), and a direct computation we get
\begin{equation}\label{th8}
\begin{array}{lll}
e_1e_3\{e_3(\kappa_1)\}-e_3e_1\{e_3(\kappa_1)\} =[e_1,e_3] \{e_3(\kappa_1)\}
=\kappa_1e_3e_3(\kappa_1)
=-\kappa_1^4-c\kappa_1^2,
\end{array}
\end{equation}
and
\begin{equation}\label{th9}
\begin{array}{lll}
e_1e_3\{e_3(\kappa_1)\}-e_3e_1\{e_3(\kappa_1)\}
=e_1\{e_3e_3(\kappa_1)\}-e_3\{e_1e_3(\kappa_1)\}\\
=-c\kappa_1^2-4c-3e_3^2(\kappa_1).\\
\end{array}
\end{equation}

Comparing  (\ref{th8}) with (\ref{th9}), we get
\begin{equation}\label{th10}
\begin{array}{lll}
3e_3^2(\kappa_1)=\kappa_1^4-4c.\\
\end{array}
\end{equation}
 Applying $e_3$ to both sides of  (\ref{th10}) and using (\ref{th6}) to
simplify the resulting equation we have
\begin{equation}\label{th12}
\begin{array}{lll}
\kappa_1(5\kappa_1^2+3c)e_3(\kappa_1)=0,
\end{array}
\end{equation}
which implies $e_3(\kappa_1)=0$. Substituting this into (\ref{th6}) and using that fact that  $c\kappa_1\neq0$ we obtain
 \begin{equation}\label{th13}
\begin{array}{lll}
\kappa_1^2=-c>0,
\end{array}
\end{equation}
 which implies that
 \begin{equation}\label{th14}
\begin{array}{lll}
p_y^2=-c,\;({\rm and\;hence})\; p_y=\pm\sqrt{-c}.
\end{array}
\end{equation}
It follows that
 \begin{equation}\label{th15}
\begin{array}{lll}
p(x,y)=\pm\sqrt{-c}y+\varphi(x),
\end{array}
\end{equation}
where $\varphi(x)$ is a an arbitrary function.\\

So we conclude that up to an isometry of the domain and/or codomain, a proper biharmonic Riemannian submersion
 $\pi:H^2(c)\times\r\to\r^2$ with $c<0$ is expressed as
\begin{equation}\label{th2}
 \pi: (\r^3, e^{2\sqrt{-c}\;y}dx^2+dy^2+dz^2) \to(\r^2,dy^2+dz^2),\;\pi(x, y, z)=(y,z).
 \end{equation}
 Thus, we obtain the proposition.
\end{proof}
\begin{remark}\label{r5}
By Proposition \ref{Pr3}, for $c\geq0$, there exists no proper biharmonic Riemannian submersion
$\pi:M^2(c)\times\r\to (N^2,h)$ no matter what $(N^2, h)$ is.
\end{remark}

\end{document}